\documentclass[a4paper, 10pt, twoside, notitlepage]{amsart}

\usepackage[utf8]{inputenc}
\usepackage{color}
\usepackage{amsmath} 
\usepackage{amssymb} 
\usepackage{amsthm}
\usepackage{geometry}
\usepackage{graphicx}
\usepackage{esint}
\usepackage[colorlinks=true,linkcolor=blue]{hyperref}

\theoremstyle{plain}
\newtheorem{thm}{Theorem}
\newtheorem{prop}{Proposition}[section]
\newtheorem{lem}[prop]{Lemma}
\newtheorem{cor}[prop]{Corollary}

\newcommand {\R} {\mathbb{R}} 
 \newcommand {\N} {\mathbb{N}}
 
\newcommand {\p} {\partial}

\newcommand {\D} {\Delta}

\DeclareMathOperator{\conv}{conv}

\DeclareMathOperator {\dist} {dist}

\DeclareMathOperator {\supp} {supp}

\DeclareMathOperator{\F} {\mathcal{F}}

\title{On Single Measurement Stability for the Fractional Calder\'on Problem}
\author{Angkana Rüland}
\address{Institut für Angewandte Mathematik, Ruprecht-Karls-Universität Heidelberg, Im Neunheimer Feld 205, 69120 Heidelberg, Germany}
\email{Angkana.Rueland@uni-heidelberg.de}

\begin{document}
\maketitle

\begin{abstract}
In this short article we complement the known single measurement uniqueness result for the fractional Calder\'on problem by a single measurement logarithmic stability estimate. To this end, we combine quantitative propagation of smallness results for the Caffarelli-Silvestre extension and a boundary doubling estimate. The latter yields control of the order of vanishing of solutions to the fractional Schrödinger equation and provides the central step in passing from the quantitative unique continuation for solutions to the logarithmic stability of the potential $q$.
\end{abstract}


\section{Introduction}

In this article we study the stability of the single measurement uniqueness result for the fractional Calder\'on problem which had been established in \cite{GRSU18}. The fractional Calder\'on problem is a nonlocal inverse problem which deals with the recovery of the potential $q $ in the equation
\begin{align}
\label{eq:nonlocal}
\begin{split}
((-\D)^s + q )u & = 0 \mbox{ in } \Omega,\\
u & = f \mbox{ in } \Omega_e,
\end{split}
\end{align}
from the ``Dirichlet-to-Neumann'' type measurement $\Lambda_q(f):=(-\D)^s f|_{\Omega_e}$.
Here and throughout the following text, for simplicity, we assume that zero is not a Dirichlet eigenvalue of the fractional Schrödinger operator under consideration. Furthermore, we suppose that $s\in (0,1)$, $(-\D)^s$ denotes the fractional Laplacian which is defined through its Fourier symbol $|\xi|^{2s}$, $\Omega_e:= \R^n \setminus \overline{\Omega}$ and $f \in \widetilde{H}^s(\Omega_e)$. This problem had been introduced and first investigated in \cite{GSU16} as a \emph{nonlocal} inverse problem modelled on the classical Calder\'on problem \cite{Calderon, U09}.

Due to its nonlocality, the fractional Calder\'on problem displays features distinguishing it from its classical, local analogue. This includes striking Runge approximation results which are valid for the nonlocal operator and have led to the solution of the uniqueness question for the fractional Calder\'on problem in \cite{GSU16} with \emph{infinitely many} measurements and partial data. These ideas can also be quantified resulting in the optimal logarithmic stability of the fractional Calder\'on problem with infinitely many measurements (and almost scaling critical regularity) \cite{RS20a, RS18}. We also refer to \cite{R19, RS20, RS19} and \cite{GFR20, DSV14} for related results and ideas. 

In contrast to the local Calder\'on problem, the fractional Calder\'on problem is already a determined problem if one only considers the case of a \emph{single} measurement. Using this observation, in \cite{GRSU18} a \emph{single measurement} uniqueness and reconstruction result of the fractional Calder\'on problem was derived based on the strong unique continuation properties for the fractional Laplacian, see \cite{FF14, Rue15, GFR19, Y17}. In this note we prove that the single measurement uniqueness result for the fractional Calder\'on problem can be complemented with a \emph{single measurement stability} result. For further information on the fractional Calder\'on problem we refer to \cite{GLX17,HL19, HL20, CLR20, RueS19,LLR20} and the survey articles \cite{S17, Rue18}. We also point to \cite{CR20} for a recent interpretation of fractional Calder\'on type problems which provides a connection to inverse Robin problems with degenerate conductivity.

\subsection{Main result and strategy}

In this article, we prove a stability result complementing the single measurement uniqueness result from \cite{GRSU18}.  Using the notation from Section \ref{sec:prelims} for the function spaces, our main result reads:

\begin{thm}
\label{thm:single_meas_stab}
Let $\Omega, W \subset \R^n$ be open, bounded, non-empty Lipschitz sets such that $\overline{\Omega} \cap \overline{W} = \emptyset$. Let $f\in \widetilde{H}^{s+\epsilon}(W)\setminus \{0\}$ for some $\epsilon>0$ and let $F>0$ be such that
\begin{align*}
\frac{\|f\|_{H^{s}(W)}}{\|f\|_{L^2(W)}} \leq F.
\end{align*}
Assume that $q_1, q_2 \in C^{0,s}(\Omega)$ with $\supp(q_j) \subset \Omega' \Subset \Omega$ satisfy the a priori bound
\begin{align*}
\|q_j\|_{C^{0,s}(\Omega)} \leq E < \infty \mbox{ for } j\in\{1,2\}.
\end{align*}
Then, there exists a modulus of continuity $\omega(t)$ with $\omega(t) \leq C |\log(C t)|^{-\gamma}$ such that
\begin{align*}
\|q_1 - q_2\|_{L^{\infty}(\Omega)} \leq \omega(\|\Lambda_{q_1} f- \Lambda_{q_2}f\|_{H^{-s}(W)}).
\end{align*}
The constants $\gamma>0$, $C>1$ only depend on $\Omega, W, s, E, F, n, \|f\|_{H^{s+\epsilon}(W)}, \dist(\Omega',\partial \Omega)$.
\end{thm}

We remark that similarly as in the case of the Calder\'on problem the logarithmic stability is optimal (possibly up to the exponent $\gamma$), see \cite{RS18} for a proof of this in the case of infinitely many measurements and \cite[Section 6.1]{GRSU18} for the optimality of logarithmic stability estimates for the UCP for the fractional Laplacian.

In order to prove the single measurement stability result, we proceed in two main steps
which are inspired by the uniqueness proof from \cite{GRSU18}: First, by the strong unique continuation properties of the fractional Laplacian, it is possible to reconstruct $u$ from the data $(f, (-\D)^s u|_{W})$ for a \emph{single}, nontrivial $f$. This step can be quantified following ideas in \cite{RS20a} which had, in parts, already been explained in \cite{GRSU18}. In a second step, we seek to exploit that if
$u : \R^n \rightarrow \R$ is known, it is possible to obtain the unknown potential $q$ as
\begin{align*}
q(x):= - \frac{(-\D)^s u(x)}{u(x)} \mbox{ for } x \in \Omega,
\end{align*}
provided there is sufficient control on the zeros of $u(x)$. In \cite{GRSU18} this was achieved by a weak unique continuation result if $q \in C^{0}(\Omega)$ and a unique continuation result from measurable sets in the case that $q\in L^{\infty}(\Omega)$ (and $s\in (1/4,1)$), respectively. The quantification of this step thus poses significant challenges, in that one needs to infer \emph{quantitative control} on the zero set of $u$. Here we follow an idea that had been introduced in \cite{S10,ASV13} in the context of non-degenerate equations and achieve the required control by deducing boundary doubling estimates. While the overall strategy follows the ideas of \cite{S10,ASV13}, we emphasise that in exploiting the Caffarelli-Silvestre extension \cite{CS07} which realises the fractional Laplacian, in general, we here no longer deal with a \emph{uniformly} elliptic but with a \emph{degenerate} elliptic equation (see \eqref{eq:CS_ext} below).

More precisely, we provide the following estimates:

\begin{itemize}
\item[(i)] We first transport information from $W$ to $\Omega$, i.e., under suitable a priori bounds for $f \neq 0$ and the a priori bound $\|q_j\|_{L^{\infty}(\Omega)}\leq M < \infty$ for the potentials $q_j$ we show that
\begin{align}
\label{eq:prop_smallness}
\begin{split}
\|u_1 - u_2\|_{H^{s}(\Omega)} &\leq C(M,\|f\|_{\widetilde{H}^{s+\epsilon}(W)}) \omega(\|(-\D)^s u_1 - (-\D)^s u_2 \|_{H^{-s}(W)}),\\
\|(-\D)^s u_1 - (-\D)^s u_2 \|_{H^{-s}(\Omega)} &\leq C(M,\|f\|_{\widetilde{H}^{s+\epsilon}(W)}) \omega(\| (-\D)^s u_1 - (-\D)^s u_2 \|_{H^{-s}(W)}),
\end{split}
\end{align}
where $\omega(t) \leq C |\log(t)|^{-\gamma}$ for some $\gamma>0$ and $\epsilon>0$.
Here $u_1, u_2$ are solutions to \eqref{eq:nonlocal} with potentials $q_1, q_2$. We discuss the derivation of these results in Section \ref{sec:smallness}. They build on quantitative unique continuation estimates which had already been derived earlier in \cite{RS20a} and \cite{GRSU18}.
\item[(ii)] As a second step and as some of the main results of the article, we deduce both \emph{bulk} and \emph{boundary doubling estimates}. The latter then leads to a controlled order of vanishing for the nonlocal problem. The rate of vanishing depends on $W,\Omega,n$, $\frac{\|f\|_{H^s(W)}}{\|f\|_{L^2(W)}}$ and on $\|q\|_{L^{\infty}(\Omega)}$: In other words, in the notation from Section \ref{sec:prelims}, we infer that for some exponent $\beta>0$ and some constant $C>1$ which depend on $\Omega, W, n, s, \|q\|_{L^{\infty}(\Omega)}$ and $\frac{\|f\|_{H^s(W)}}{\|f\|_{L^2(\Omega)}}$ (but not on $u|_{\Omega}$) and for all radii $r\in (0,r_0)$ and $x_0 \in \Omega$ such that $\dist(x_0,\partial \Omega)\geq 2r_0>0$ we have for solutions $u$ to \eqref{eq:nonlocal}
\begin{align*}
\| u\|_{L^2(B_{r}'(x_0))} \geq C r^{\beta}.
\end{align*} 
These estimates will be derived in Section \ref{sec:doubling}.
\end{itemize}
Finally, invoking an interpolation type lemma from \cite{ASV13}, the combination of the two steps (i), (ii) implies the desired overall single measurement stability result.

As in \cite{RS20a} in our analysis and derivation of the quantitative unique continuation estimates we heavily rely on the connection between the fractional Laplacian and the Caffarelli-Silvestre extension. Indeed, using the Caffarelli-Silvestre extension \cite{CS07}, the equation \eqref{eq:nonlocal} can also be viewed as a problem in the upper half-plane $\R^{n+1}_+$: For $s\in (0,1)$ we consider the following equation
\begin{align}
\label{eq:CS_ext}
\begin{split}
\nabla \cdot x_{n+1}^{1-2s} \nabla \tilde{u} & = 0 \mbox{ in } \R^{n+1}_+,\\
\tilde{u} & = u \mbox{ on } \R^n,
\end{split}
\end{align}
where $u$ is the solution to \eqref{eq:nonlocal} and where we identified $\R^n \times \{0\}$ with $\R^n$. Moreover, in a weak sense, $(-\D)^s u = c_s \lim\limits_{x_{n+1}\rightarrow 0} x_{n+1}^{1-2s} \p_{n+1} \tilde{u}$ for some constant $c_s \neq 0$. In particular, $\tilde{u}|_{W} =f$. In the sequel, we will often switch between the interpretations \eqref{eq:nonlocal} and \eqref{eq:CS_ext}. The Caffarelli-Silvestre interpretation of the problem \eqref{eq:nonlocal} lies at the heart of both of our main steps (i), (ii) outlined above, as it allows us to make use of techniques for (degenerate) elliptic, local equations in deducing quantitative unique continuation arguments. 

\subsection{Outline of the article}
The remainder of the article is organised as follows: After briefly recalling our notation and some preliminary facts on the Caffarelli-Silvestre extension in Section \ref{sec:prelims}, in Section \ref{sec:smallness} we explain the derivation of the estimates from Step (i). Here we make use of the earlier bounds from \cite{GRSU18} and \cite{RS20a}. In Section \ref{sec:doubling} -- which constitutes the central estimates of the article -- we deduce both bulk and boundary doubling estimates. Related ideas had already been used in \cite{Rue17} but in the context of compact manifolds. In order to deal with the non-compactness of our setting, we make use of the a priori estimates for solutions to \eqref{eq:nonlocal} and the bounds for $f$ in terms of $\frac{\|f\|_{H^s(W)}}{\|f\|_{L^2(W)}}$. Finally, in Section \ref{sec:single_meas} we prove the desired stability result.

\section{Notation and some Preliminaries}
\label{sec:prelims}

In this section, we recall some of the notation for fractional Sobolev spaces and some basic estimates for the Caffarelli-Silvestre extension.

For $Q \subset \R^n$ open and $s\in \R$, we will use the notation $H^{s}(Q)$ to denote the (fractional) Sobolev space given by
\begin{align*}
H^{s}(Q):= \{f : \Omega \rightarrow \R: \ \|f\|_{H^s(Q)}<\infty \},
\end{align*}
where $\|f\|_{H^{s}(Q)}:= \inf\{\|F\|_{H^s(\R^n)}: \ F|_{Q} = f\}$. The whole-space norm $\|\cdot\|_{H^s(\R^n)}$ is defined by its Fourier multiplier:  
\begin{align*}
\|f\|_{H^s(\R^n)}:= \left(\int\limits_{\R^n}
(1+|\xi|^2)^s |\mathcal{F}(f)(\xi)|^2 d\xi \right)^{\frac{1}{2}},
\end{align*}
where $\F(f)$ denotes the Fourier transform of $f$. 

We will also be working with two further classes of Sobolev functions:
The Sobolev functions with compact support in a closed set $\overline{Q} \subset \R^n$ will be denoted by $H^{s}_{\overline{Q}}$. For $\Omega \subset \R^n$ open, the completion of $C_c^{\infty}(\Omega)$ in $H^s(\R^n)$ will be denoted by $\widetilde{H}^s(\Omega)$. In the case of Lipschitz domains these two spaces are equal for any $s\in \R$.

Finally, since we will strongly rely on properties of solutions to \eqref{eq:nonlocal} and on their Caffarelli-Silvestre extensions \eqref{eq:CS_ext}, we recall the basic a priori estimates for these. First, energy estimates (see \cite[Lemma 2.3]{GSU16}) yield that for bounded, open sets $\Omega\subset \R^n$, $s\in (0,1)$ and potentials $q \in L^{\infty}(\Omega)$ for solutions to \eqref{eq:nonlocal} (where we assume that zero is not a Dirichlet eigenvalue of the fractional Schrödinger operator) we have
\begin{align}
\label{eq:apriori_nonloc}
\|u|_{\Omega}\|_{H^s_{\overline{\Omega}}} \leq C \|f\|_{H^s(W)}.
\end{align} 
Since $\Omega$ and $W$ are disjoint, in the following sections, with a slight abuse of notation, we will also write $\|u\|_{H^{s}_{\overline{\Omega}}}$ instead of $\|u|_{\Omega}\|_{H^s_{\overline{\Omega}}}$ for solutions $u$ to \eqref{eq:nonlocal}.
For the Caffarelli-Silvestre extension $\tilde{u}$ of a solution $u$ to \eqref{eq:nonlocal}, it holds that
\begin{align}
\label{eq:apriori_grad}
\|x_{n+1}^{\frac{1-2s}{2}} \nabla \tilde{u}\|_{L^2(\R^{n+1}_+)} \leq C \|u\|_{H^s(\R^{n})}.
\end{align}
Hence, by a combination of the previous two bounds we then also obtain
\begin{align}
\label{eq:aux_3}
\|x_{n+1}^{\frac{1-2s}{2}} \nabla \tilde{u}\|_{L^2(\R^{n+1}_+)} \leq C \|f\|_{H^s(W)},
\end{align}
if $u$ is a solution to \eqref{eq:nonlocal} with $q \in L^{\infty}(\Omega)$.

Furthermore, Poincar\'e inequalities are available for the Caffarelli-Silvestre extension with zero boundary conditions on $\R^n$ (see for instance \cite[Lemma 4.7]{RS20a}). If $\R^n \setminus (\overline{\Omega}\cup \overline{W}) $
is an open set and $\Omega$, $W$ are bounded, open sets, for a bounded set $K \subset \R^{n+1}_+$ it therefore also holds that 
\begin{align}
\label{eq:apriori_L2}
\|x_{n+1}^{\frac{1-2s}{2}}  \tilde{u}\|_{L^2(K)} \leq C(W,\Omega,K)\|x_{n+1}^{\frac{1-2s}{2}} \nabla \tilde{u}\|_{L^2(\R^n)} \leq  C(W,\Omega,K)\|f\|_{H^s(W)}.
\end{align}
Indeed, \eqref{eq:apriori_L2} follows from a two fold application of Poincar\'e's inequality: Assuming without loss of generality that $K\subset Q \times [c,d]:=[a_1,b_1]\times \cdots \times [a_n, b_n]\times [c,d]$ for $a_j,b_j,c,d\in \R$ with $d>c\geq 0$, $b_j>a_j$, we first apply the Poincar\'e inequality (for instance from \cite[Lemma 4.7]{RS20a}) in the normal direction starting from a boundary cube $Q':=[a'_1,b'_1]\times \cdots \times [a_n',b_n'] \subset \R^n\setminus (\overline{\Omega} \cup \overline{W})$ where $u$ vanishes. This yields
\begin{align}
\label{eq:Poinc_aux1}
\|x_{n+1}^{\frac{1-2s}{2}} \tilde{u}\|_{L^2(Q'\times [c,d])}\leq C(c,d,Q')\|x_{n+1}^{\frac{1-2s}{2}} \nabla \tilde{u}\|_{L^2(Q'\times [0,d])}).
\end{align}
In a second step, we apply a tangential Poincar\'e estimate (which follows from the fundamental theorem of calculus and a density argument) to bound
\begin{align}
\label{eq:Poinc_aux2}
\|x_{n+1}^{\frac{1-2s}{2}} \tilde{u}\|_{L^2(Q\times [c,d])}
\leq C(Q,Q')(\|x_{n+1}^{\frac{1-2s}{2}} \tilde{u}\|_{L^2(Q'\times [c,d])} + \|x_{n+1}^{\frac{1-2s}{2}} \nabla \tilde{u}\|_{L^2((\conv(Q'\cup Q))\times [c,d])}),
\end{align}
where $\conv(Q'\cup Q)$ denotes the convex hull of $Q'\cup Q$. Combining \eqref{eq:Poinc_aux1}, \eqref{eq:Poinc_aux2} and \eqref{eq:aux_3} then implies \eqref{eq:apriori_L2}.

For balls in $\R^n$ or $\R^{n+1}_+:=\{x=(x',x_{n+1})\in \R^{n+1}: \ x_{n+1}\geq 0\}$ we use the following notation
\begin{align*}
B_{r}'(x_0)&:=\{x\in \R^n: \ |x-x_0|< r\}, \
B_{r}^+((x_0,0)):=\{x \in \R^{n+1}_+: \ |x-(x_0,0)|<r\}, \\
B_{r}(\bar{x})&:= \{x \in \R^{n+1}_+: \ |x-\bar{x}|<r\},
\end{align*}
if $x_0 \in \R^n$ and $\bar{x} \in \R^{n+1}_+$ is such that $B_{r}(\bar{x}) \subset (\R^{n+1}_+)^{\circ}$. If $x_0\in \R^n$ is the origin, we also simply omit the centre point in the notation and write $B_r'$ and $B_{r}^+$, respectively.

\section{Transportation of Smallness}
\label{sec:smallness}
In this section we provide the proof of step (i) from the introduction.
To this end, we first recall the following propagation of smallness estimate from \cite{GRSU18}:

\begin{prop}[Proposition 6.1 in \cite{GRSU18}]
\label{prop:stab_UCP}
Let $\Omega, W \subset \R^n$ with $n\geq 1$ be bounded, open, non-empty Lipschitz domains with $\overline{\Omega} \cap \overline{W}= \emptyset$. Let $s\in (0,1)$ and $s' \in (0,s)$ and assume that $q\in L^{\infty}(\Omega)$ with $\|q\|_{L^{\infty}(\Omega)} \leq M<\infty$. Let $v\in H^s_{\overline{\Omega}}$.
Then there exist constants $C,\sigma>0$ (depending only on $\Omega, W, n, s, s',M$) such that for any $\tilde{E}>0$ one has
\begin{align*}
\|v\|_{H^{s'}_{\overline{\Omega}}} \leq C \frac{\tilde{E}}{\left| \log\left(C \frac{\tilde{E}}{\|(-\D)^s v\|_{H^{-s}(W)}} \right) \right| ^{\sigma}},
\end{align*}
whenever $\|v\|_{H^{s}_{\overline{\Omega}}} \leq \tilde{E}$.
\end{prop}

Applying this to the function $v:=u_1 -u_2 \in H^s_{\overline{\Omega}}$, where $u_1, u_2$ are solutions to \eqref{eq:nonlocal} with $u_1|_{W}= f= u_2|_{W}$, we obtain a logarithmic stability result for the recovery of $u_1-u_2$ from the measurements $((-\D)^s u_1 - (-\D)^s u_2)|_{W}$.

\begin{cor}
\label{cor:stabilityu}
Let $\Omega, W \subset \R^n$ with $n\geq 1$ be bounded, open, non-empty Lipschitz domains with $\overline{\Omega} \cap \overline{W}= \emptyset$. Let $s\in (0,1)$ and $s' \in (0,s)$. 
Let $u_1, u_2$ be solutions to \eqref{eq:nonlocal} with data $f \in H^{s}_{\overline{W}} \setminus \{0\}$ and such that for some $M>0$
\begin{align*}
\|q_1\|_{L^{\infty}(\Omega)} , \|q_2 \|_{L^{\infty}(\Omega)} \leq M < \infty.
\end{align*}
Then there exist constants $C,\sigma>0$ (depending only on $\Omega, W, n,s, s', M$) such that 
\begin{align*}
\|u_1- u_2\|_{H^{s'}_{\overline{\Omega}}} \leq \frac{C(M) \|f\|_{H^s(W)}}{\left|\log\left(\frac{C(M)\|f\|_{H^s(W)}}{\|(\Lambda_{1}-\Lambda_2)f\|_{H^{-s}(W)}} \right)\right|^{\sigma}}.
\end{align*}
Here $\Lambda_j (f):= \Lambda_{q_j}(f) = (-\D)^s u_j|_{W} $ for $j\in \{1,2\}$.
\end{cor}

\begin{proof}
We apply Proposition \ref{prop:stab_UCP} to the function $v := u_1 -u_2 \in H^{s}_{\overline{\Omega}}$ and note that $(-\D)^s v(x) = (\Lambda_1 - \Lambda_2)f(x)$. It thus remains to argue that 
\begin{align*}
\|v\|_{H^{s}_{\overline{\Omega}}}
\leq C(M)\|f\|_{H^s(W)}.
\end{align*}
This however follows by the triangle inequality and from the a priori estimates for solutions to the equation \eqref{eq:nonlocal}, see \eqref{eq:apriori_nonloc}:
\begin{align*}
\|v\|_{H^{s}_{\overline{\Omega}}} \leq \|u_1|_{\Omega}\|_{H^s_{\overline{\Omega}}} + \|u_2|_{\Omega}\|_{H^s_{\overline{\Omega}}}
\leq C(M)\|f\|_{H^s(W)}.
\end{align*}
\end{proof}

By interpolation the estimate from Corollary \ref{cor:stabilityu} can be upgraded to an $H^{s}_{\overline{\Omega}}$ stability estimate for $(u_1- u_2)|_{\Omega}$. By using the equation for $u_1, u_2$ in its Caffarelli-Silvestre extended version, this also directly implies an $H^{-s}(\Omega)$ estimate for $\lim\limits_{x_{n+1}\rightarrow 0} x_{n+1}^{1-2s} \p_{n+1} \tilde{u}_1 - \lim\limits_{x_{n+1}\rightarrow 0} x_{n+1}^{1-2s} \p_{n+1} \tilde{u}_2$.

\begin{prop}
\label{prop:stab}
Let $\Omega, W \subset \R^n$ with $n\geq 1$ be bounded, open, non-empty Lipschitz domains with $\overline{\Omega} \cap \overline{W}= \emptyset$.
Let $u_1, u_2$ be solutions to \eqref{eq:nonlocal} with data $f \in H^{s+\epsilon}_{\overline{W}} \setminus \{0\}$ for some $\epsilon>0$ and such that for some $M>0$
\begin{align*}
\|q_1\|_{L^{\infty}(\Omega)} , \|q_2 \|_{L^{\infty}(\Omega)} \leq M.
\end{align*}
Then there exist constants $C,\sigma>0$ (depending only on $\Omega, W, n,s, M,\epsilon$) such that 
\begin{align}
\label{eq:prop_small_norms}
\begin{split}
\|u_1- u_2\|_{H^{s}_{\overline{\Omega}}} &\leq  \frac{C(M) \|f\|_{H^{s+\epsilon}(W)}}{\left|\log\left(\frac{C(M)\|f\|_{H^{s+\epsilon}(W)}}{\|(\Lambda_{1}-\Lambda_2)f\|_{H^{-s}(W)}} \right)\right|^{\sigma}},\\
\|\lim\limits_{x_{n+1}\rightarrow 0} x_{n+1}^{1-2s} \p_{n+1}\tilde{u}_1- \lim\limits_{x_{n+1}\rightarrow 0} x_{n+1}^{1-2s} \p_{n+1} \tilde{u}_2\|_{H^{-s}(\Omega)} &\leq  \frac{C(M) \|f\|_{H^{s+\epsilon}(W)}}{\left|\log\left( \frac{C(M)\|f\|_{H^{s+\epsilon}(W)}}{\|(\Lambda_{1}-\Lambda_2)f\|_{H^{-s}(W)}} \right)\right|^{\sigma}}.
\end{split}
\end{align}
\end{prop}

\begin{proof}
Since $f\in H^{s+\epsilon}_{\overline{W}}$ for some $\epsilon>0$, we may without loss of generality assume that $\epsilon \in (0,\min\{\frac{s}{2},\frac{1}{2}\})$ and apply Vishik-Eskin estimates \cite{VE65}  (see also \cite[Lemma 6.2]{RS20a}). These entail that
\begin{align*}
\|u_j \|_{H^{s+\epsilon}_{\overline{\Omega}}}
\leq C \|f\|_{H^{s+\epsilon}(W)}, \ j \in \{1,2\},
\end{align*}
where $C=C(n,\Omega, W, \|q\|_{L^{\infty}(\Omega)},s,\epsilon)>0$.

In order to derive the first estimate in \eqref{eq:prop_small_norms}, we next use interpolation together with the estimate from Corollary \ref{cor:stabilityu}:
\begin{align*}
\|u_1 -u_2\|_{H^{s}_{\overline{\Omega}}} 
&\leq C\|u_1 -u_2\|_{H^{s+\epsilon}_{\overline{\Omega}}}^{\frac{1}{2}} \|u_1 - u_2\|_{H^{s-\epsilon}_{\overline{\Omega}}}^{\frac{1}{2}}\\
& \leq C \|f\|_{H^{s+\epsilon}(W)}^{\frac{1}{2}} \|u_1 - u_2\|_{H^{s-\epsilon}_{\overline{\Omega}}}^{\frac{1}{2}}
\leq \frac{C(M) \|f\|_{H^{s+\epsilon}(W)}}{\left|\log\left(\frac{C(M)\|f\|_{H^{s+\epsilon}(W)}}{\|(\Lambda_{1}-\Lambda_2)f\|_{H^{-s}(W)}} \right)\right|^{\tilde{\sigma}}}.
\end{align*}
Here $\tilde{\sigma}=  \frac{\sigma}{2}$, where $\sigma>0$ is the constant from Corollary \ref{cor:stabilityu}.

Regarding the second estimate in \eqref{eq:prop_small_norms}, we note that by virtue of the Caffarelli-Silvestre extension and the compact support of $u_1-u_2$ we obtain
\begin{align*}
\|\lim\limits_{x_{n+1}\rightarrow 0} x_{n+1}^{1-2s} \p_{n+1}\tilde{u}_1- \lim\limits_{x_{n+1}\rightarrow 0} x_{n+1}^{1-2s} \p_{n+1} \tilde{u}_2 \|_{H^{-s}(\Omega)}
& = c_s \|(-\D)^s u_1 - (-\D)^s u_2\|_{H^{-s}(\Omega)}\\
& \leq c_s \| u_1 - u_2 \|_{H^{s}_{\overline{\Omega}}}.
\end{align*}
Now, using the estimate from the first part of \eqref{eq:prop_small_norms} also implies the second bound stated in \eqref{eq:prop_small_norms}.
\end{proof}

\section{Doubling Estimates}
\label{sec:doubling}
Let us now turn to Step (ii) in the programme layed out in the introduction.
Here the main novelty is the discussion of the \emph{boundary doubling estimates} (see also \cite{BL15, S10, ASV13} for related results for Steklov-type operators). Related estimates for fractional Schrödinger operators had earlier been derived in \cite{Rue17} but on compact manifolds. After deducing an auxiliary Caccioppoli estimate in Section \ref{sec:Cacc}, in order to infer the doubling estimates, we argue in two steps: First, we produce \emph{bulk} doubling estimates in Section \ref{sec:bulkdoubl}, then in a second step, in Section \ref{sec:boundarydoubl}, we derive \emph{boundary} doubling estimates from these. The main point in the estimates is that the exponent $\beta$ in (ii) is not allowed to depend on $u|_{\Omega}$ directly, but only on the boundary data $f$ (it will in fact depend on the oscillation of the boundary data) and the other a priori data.

\subsection{Caccioppoli inequalities}
\label{sec:Cacc}

Since we will often use this in our proof of the doubling estimates, we recall the Caccioppoli estimate from \cite[Lemma 5.1]{RW19}.

\begin{prop}
\label{prop:Cacc}
Let $\Omega, W \subset \R^n$ with $n\geq 1$ be bounded, open, non-empty Lipschitz domains with $\overline{\Omega} \cap \overline{W}= \emptyset$. Let $s\in (0,1)$ and let $u \in H^s(\R^n)$ be a solution to \eqref{eq:nonlocal}. Let $\tilde{u}$ denote its Caffarelli-Silvestre extension satisfying \eqref{eq:CS_ext}. Assume that $x_0 \in \Omega$ and that $r \in (0,r_0)$ where $r_0>0$ is such that $\dist(x_0, \partial \Omega) \geq 4 r_0>0$. Then, there exists a constant $C>0$ depending on $n,s,\Omega,r_0$ such that for all $r\in (0,r_0)$
\begin{align*}
\|x_{n+1}^{\frac{1-2s}{2}} \nabla \tilde{u}\|_{L^2(B_{r}^+(x_0))}
\leq C (1 + \|q\|_{L^{\infty}(\Omega)}^{\frac{1}{2s}})r^{-1}\|x_{n+1}^{\frac{1-2s}{2}} \tilde{u}\|_{L^2(B_{2 r}^+(x_0))}.
\end{align*}
\end{prop}

For the convenience of the reader we recall the proof of this result. It relies on a combination of trace estimates and the equation for $u$.

\begin{proof}
We first note that by scaling and translating, we may assume that $r=1$ and that $x_0 = 0$.\\ 

\emph{Step 1: Proof of the result in the case that $\|q\|_{L^{\infty}(\R^n)}$ is sufficiently small.}
As in \cite{RW19} we first prove the result if $\|q\|_{L^{\infty}(\R^n)} \leq \delta$ for some $\delta \in (0,\delta_0)$ sufficiently small. Testing the bulk Caffarelli-Silvestre equation with $\eta^2 \tilde{u}$ where $\eta$ is a cut-off function with $\supp(\eta) \subset B_1^+$ we infer that
\begin{align}
\label{eq:Cacc_1}
\|x_{n+1}^{\frac{1-2s}{2}} \eta \nabla \tilde{u}\|_{L^2(B_1^+)}
\leq C_1 (\|x_{n+1}^{\frac{1-2s}{2}} |\nabla \eta| \tilde{u} \|_{L^2(B_1^+)} + \| q \|_{L^{\infty}(\R^n)}^{\frac{1}{2}}\|\eta u\|_{L^2(B_1')}).
\end{align}
By trace estimates and the compact support of $\eta$ we have that 
\begin{align*}
\|\eta u\|_{L^2(B_1')} \leq C \|x_{n+1}^{\frac{1-2s}{2}} \nabla (\eta \tilde{u})\|_{L^2(B_1^+)} \leq C_2 (\|x_{n+1}^{\frac{1-2s}{2}} \eta (\nabla \tilde{u})\|_{L^2(B_1^+)} + \|x_{n+1}^{\frac{1-2s}{2}} (\nabla \eta) \tilde{u}\|_{L^2(B_1^+)}  ).
\end{align*}
Inserting this back into \eqref{eq:Cacc_1} and using the smallness of $\|q\|_{L^{\infty}(\R^n)}\leq \delta$, allows us to conclude the desired estimate
\begin{align}
\label{eq:Cacc_norm}
\|x_{n+1}^{\frac{1-2s}{2}} \nabla \tilde{u}\|_{L^2(B_{2/3}^+)}
\leq C \|x_{n+1}^{\frac{1-2s}{2}} \tilde{u} \|_{L^2(B_1^+)}.
\end{align}
if $\delta_0^{\frac{1}{2}} \leq \frac{1}{2 C_1 C_2}$.\\

\emph{Step 2: Rescaling, reduction to Step 1.} Since the problem is subcritical, it is possible to rescale the problem so that we may apply Step 1. More precisely, the function $\tilde{u}_{\delta,x_0}:=\tilde{u}(x_0 + \delta x)$ for $x_0 \in \R^n \times \{0\}$ arbitrary solves an equation of the type \eqref{eq:CS_ext} but with $q$ replaced by $q_{\delta}(x):= \delta^{2s} q(\delta x + x_0)$. As a consequence, for $\delta \leq \left(\frac{\delta_0}{\|q\|_{L^{\infty}(\R^n)}} \right)^{\frac{1}{2s}} $, we in particular obtain that $\|q_{\delta}\|_{L^{\infty}(\R^n)} \leq \delta_0$ with $\delta_0>0$ as in Step 1 of the proof. Thus, \eqref{eq:Cacc_norm} is applicable to $u_{\delta,x_0}$. After rescaling this implies that
\begin{align*}
\|x_{n+1}^{\frac{1-2s}{2}} \nabla \tilde{u}\|_{L^2(B_{2\delta/3}^+(x_0))} \leq C \left( \frac{\|q\|_{L^{\infty}(\R^n)}}{\delta_0} \right)^{\frac{1}{2s}} \|x_{n+1}^{\frac{1-2s}{2}} \tilde{u}\|_{L^2(B_{\delta}^+(x_0))}.
\end{align*} 
Covering $B_{1}'(x_0)$ by balls of radius $2\delta/3$ which have a controlled (dimension-dependent) overlap then implies that
\begin{align*}
\|x_{n+1}^{\frac{1-2s}{2}} \nabla \tilde{u}\|_{L^2(B_{1}'(x_0)\times [0,\delta/2))} \leq C \left( \frac{\|q\|_{L^{\infty}(\R^n)}}{\delta_0} \right)^{\frac{1}{2s}} \|x_{n+1}^{\frac{1-2s}{2}} \tilde{u}\|_{L^2(B_{2}'(x_0)\times [0,\delta))}.
\end{align*} 
Finally in $B_1'(x_0)\times (\delta/2,1]$ we apply a further (controlled) covering argument and Caccioppoli's inequality without boundary contributions. Combining all of this yields the desired estimate.
\end{proof}

\subsection{The bulk doubling inequality}
\label{sec:bulkdoubl}
In this section, we begin our derivation of the central doubling estimates. To this end, we first prove a bulk doubling estimate.

\begin{thm}
\label{thm:doubl_bulk}
Let $n\geq 1$ and $s \in (0,1)$. Let $\Omega, W \subset \R^n$ be bounded, open, non-empty Lipschitz domains with $\overline{\Omega} \cap \overline{W}= \emptyset$. Then, there exists a constant $C>1$ depending only on $\Omega,n , \frac{\|f\|_{H^s(W)}}{\|f\|_{L^2(W)}}, \|q\|_{L^{\infty}(\Omega)}, s, W$ such that for any $x_0 \in \Omega$, $r \in (0,r_0)$ with $r_0 = \frac{\dist(x_0,\partial \Omega)}{10}$ and for any $\tilde{u}$ solving \eqref{eq:CS_ext} with data $f \in \widetilde{H}^{s}(W)\setminus \{0\}$ it holds
\begin{align*}
\|x_{n+1}^{\frac{1-2s}{2}} \tilde{u}\|_{L^2(B_{2r}^+(x_0))} \leq C \|x_{n+1}^{\frac{1-2s}{2}} \tilde{u}\|_{L^2(B_r^+(x_0))}.
\end{align*}
\end{thm}

In order to achieve this, we rely on the Carleman inequality from \cite[Appendix A]{GRSU18} and propagation of smallness estimates. We will also rely on a number of auxiliary results which we present first.

\subsubsection{Auxiliary results} 

We begin by discussing bounds showing the ``persistence of largeness'' which are used in our propagation of smallness estimates:

\begin{lem}
\label{lem:control_f}
Let $n\geq 1$ and $s\in (0,1)$. Let $W, \Omega \subset \R^n$ be open, bounded, non-empty Lipschitz sets with $\overline{\Omega}\cap \overline{W} = \emptyset$. Let $\tilde{u}$ be a solution to \eqref{eq:CS_ext} with $f\neq 0$. Then there exists a constant $C>1$ depending on $\Omega, W, n,s$ such that for any $h\in (0,1)$
\begin{align*}
\|x_{n+1}^{\frac{1-2s}{2}} \tilde{u}\|_{L^2(W \times [h,1])} \geq  \left( \frac{1}{C} F^{-\frac{1}{s}} - h \right) \|f\|_{H^s(W)} - C_s h^{1-s} \|f\|_{L^2(W)},
\end{align*}
where, as above, $F= \frac{\|f\|_{H^s(W)}}{\|f\|_{L^2(W)}}$.
\end{lem}

\begin{proof}
The claim follows from trace estimates in a neighbourhood of $W$ and lower bounds on $\|f\|_{L^2(W)}$. 
More precisely, using the trace estimate from \cite[Lemma 2.5]{CR20}, we obtain that for all $\mu \geq \mu_0 \geq 1$ we have
\begin{align}
\label{eq:interpol_eq}
\|f\|_{L^2(W)} \leq C (\mu^{-s} \|x_{n+1}^{\frac{1-2s}{2}} \nabla \tilde{u}\|_{L^2(W \times [0,1])} + \mu^{1-s} \|x_{n+1}^{\frac{1-2s}{2}} \tilde{u}\|_{L^2(W \times [0,1])} ).
\end{align}
By Poincar\'e's inequality and the zero boundary conditions in $\R^n \setminus (\overline{W}\cup \overline{\Omega})$ (see the proof of \eqref{eq:apriori_L2}), there exists a constant $C>0$ (depending only on $\Omega, n,s,W$ and $\mu_0$) such that
\begin{align*}
C\frac{\|x_{n+1}^{\frac{1-2s}{2}} \nabla \tilde{u}\|_{L^2(2W \times [0,1])}}{\|x_{n+1}^{\frac{1-2s}{2}} \tilde{u}\|_{L^2(W \times [0,1])}}\geq \mu_0.
\end{align*}
Here, with slight abuse of notation, we have set 
\begin{align*}
2W:=\{x\in \R^n: \ \dist(x,W)\leq \frac{1}{2} \dist(W,\Omega)\}.
\end{align*}
Inserting this into \eqref{eq:interpol_eq} and using that $\|x_{n+1}^{\frac{1-2s}{2}} \nabla \tilde{u}\|_{L^2(\R^{n+1}_+)} \leq C \|f\|_{H^s(W)}$, we obtain
\begin{align*}
\|f\|_{L^2(W)} 
&\leq C \|x_{n+1}^{\frac{1-2s}{2}} \nabla \tilde{u} \|_{L^2(2W \times [0,1])}^{1-s} \|x_{n+1}^{\frac{1-2s}{2}} \tilde{u} \|_{L^2(W\times [0,1])}^s\\
& \leq C \|f\|_{H^s(W)}^{1-s} \|x_{n+1}^{\frac{1-2s}{2}} \tilde{u}\|_{L^2(W \times [0,1])}^{s}.
\end{align*}
Setting $F= \frac{\|f\|_{H^s(W)}}{\|f\|_{L^2(W)}}$, this is equivalent to
\begin{align}
\label{eq:low1}
F^{-\frac{1}{s}} \|f\|_{H^s(W)} \leq C \|x_{n+1}^{\frac{1-2s}{2}} \tilde{u} \|_{L^2(W \times [0,1])}.
\end{align}
This in particular implies a lower bound on $\|x_{n+1}^{\frac{1-2s}{2}} \tilde{u}\|_{L^2(W \times [0,1])}$ in terms of the measured data $f$ and the a priori control $F$, given in terms of the oscillation of the data $f$. 

We next observe that by an approximation argument and the fundamental theorem of calculus
\begin{align}
\label{eq:low2}
\int\limits_0^{h}\int\limits_{W} t^{1-2s}|\tilde{u}(x,t)|^2 dx dt 
&\leq h^{2} \int\limits_{0}^h \int\limits_{W} t^{1-2s} |\p_t \tilde{u}(x,t)|^2 dx dt + C_s h^{2-2s} \int\limits_{W} |\tilde{u}(x,0)|^2 dx.
\end{align}
Indeed, this follows by noting that for $C^1$ functions
\begin{align*}
|\tilde{u}(x,t)|^2 &\leq C \left( |\tilde{u}(x,t)|^2 + \left(\int\limits_{0}^h |\p_t \tilde{u}(x,\tau)|d\tau \right)^2 \right)\\
&\leq  C \left( |\tilde{u}(x,t)|^2 + C_s h^{2s}\int\limits_{0}^h \tau^{1-2s} |\p_t \tilde{u}(x,\tau)|^2d\tau \right).
\end{align*}
Now multipyling this with $t^{1-2s}$ for $t>0$ and integrating in $W\times (0,h)$ implies \eqref{eq:low2}.

Finally, invoking \eqref{eq:low2},
we infer that the mass in $W \times [h,1]$ is bounded below if $h$ is chosen properly depending on the mass of $f$.
Indeed, combining \eqref{eq:low1}, \eqref{eq:low2} and \eqref{eq:apriori_grad}, we obtain
\begin{align*}
\|x_{n+1}^{\frac{1-2s}{2}} \tilde{u} \|_{L^2(W \times [h,1])}
& \geq \frac{1}{2}\left( \|x_{n+1}^{\frac{1-2s}{2}} \tilde{u}\|_{L^2(W \times [0,1])} - \left\|x_{n+1}^{\frac{1-2s}{2}} \tilde{u} \right\|_{L^2(W \times [0,h])} \right)\\
&\geq \frac{1}{C} F^{- \frac{1}{s}} \|f\|_{H^s(W)} - h \|x_{n+1}^{\frac{1-2s}{2}}\p_{n+1} \tilde{u}\|_{L^2(\R^{n+1}_+)} - C_s h^{1-s} \|u\|_{L^2(W)}\\
& \geq \left( \frac{1}{C} F^{-\frac{1}{s}} - h \right) \|f\|_{H^s(W)} - C_s h^{1-s} \|f\|_{L^2(W)},
\end{align*}
which proves the desired result.
\end{proof}

Using this and propagation of smallness arguments, we control the quotients of mass in balls of uniform size. 

\begin{lem}[Annulus bounds]
\label{lem:annulus}
Let $n\geq 1$, $s\in (0,1)$ and let $R>0$ be a fixed radius. Let $\Omega, W \subset \R^n$ be bounded, open, non-empty Lipschitz domains with $\overline{\Omega} \cap \overline{W}= \emptyset$. Let $\tilde{u}$ be a solution to \eqref{eq:CS_ext} with data $f\in \widetilde{H}^s(W)\setminus \{0\}$. Then there exist constants $C>1$, $\gamma>1$ depending on $W,n,\Omega, R,s, \|q\|_{L^{\infty}(\Omega)}$ such that 
\begin{align*}
\frac{\|x_{n+1}^{\frac{1-2s}{2}}\tilde{u}\|_{L^2(B_{2R}^+)}}{\|x_{n+1}^{\frac{1-2s}{2}} \tilde{u}\|_{L^2(B_R^+ \setminus B_{R/2}^+)}} \leq C \left(\frac{\|f\|_{H^s(W)}}{\|f\|_{L^2(W)}} \right)^{\gamma}.
\end{align*}
\end{lem}

\begin{proof}
We argue in three steps:

\emph{Step 1: Upper bounds for $\tilde{u}$.}
First, we deduce upper bounds for $\|x_{n+1}^{\frac{1-2s}{2}} \tilde{u}\|_{L^2(B_{2R}^+)}$. Using a priori estimates for the Caffarelli-Silvestre extension on the one hand, and a priori bounds for the problem \eqref{eq:nonlocal} on the other hand (see \eqref{eq:apriori_grad}), we have that 
\begin{align*}
\|x_{n+1}^{\frac{1-2s}{2}} \nabla \tilde{u}\|_{L^2(B_{2R}^+)} \leq C\|f\|_{H^s(W)}.
\end{align*}
Using the zero boundary conditions for $u$ on $\R^n \setminus (\overline{W} \cup \overline{\Omega}) $, we then also infer the desired $L^2$ bound by Poincar\'e's inequality (see \eqref{eq:apriori_L2}):
\begin{align}
\label{eq:step1_fin}
\|x_{n+1}^{\frac{1-2s}{2}} \tilde{u}\|_{L^2(B_{2R}^+)} \leq C_R \|f\|_{H^s(W)}.
\end{align}

\emph{Step 2: Lower bounds for $\|x_{n+1}^{\frac{1-2s}{2}} \tilde{u} \|_{L^2(B_{R}^+ \setminus B_{R/2}^+)}$.}
We deduce a lower bound for $\|x_{n+1}^{\frac{1-2s}{2}} \tilde{u} \|_{L^2(B_{R}^+\setminus B_{R/2}^+)}$ by using a chain of balls argument, connecting a ball $B_{h_0/6}(\hat{x}) \subset B_{R}^+\setminus B_{R/2}^+$ to a ball $B_{h_0/6}(\bar{x}) \subset W \times [h_0,1]$, where $h_0>0$ is chosen appropriately.\\

\emph{Step 2a. Balls of non-trivial mass in $W \times [h_0,1]$.} We claim that there exist constants $h_0>0$, $r_0>0$ depending on $F, W$ such that for all $r\in (0,r_0)$ there is a ball $B_{r}(\bar{x})$ with $\bar{x}\in W \times [h_0,1]$ and
\begin{align}
\label{eq:bound_balls}
\|x_{n+1}^{\frac{1-2s}{2}} \tilde{u} \|_{L^2(B_{r}(\bar{x}))} \geq C(W,F,s) r^{n}\|f\|_{H^{s}(W)}.
\end{align}
The argument for this follows by choosing $h>0$ in Lemma \ref{lem:control_f} appropriately and by a compactness and averaging argument. More precisely, setting $2C_0:=\frac{F^{-\frac{1}{s}}}{C}>0$, we first choose $h_0 >0$ such that
\begin{align*}
2C_0 - h- C_s h^{1-s}\geq C_0.
\end{align*}
Hence, Lemma \ref{lem:control_f} yields that we have
\begin{align*}
\|x_{n+1}^{\frac{1-2s}{2}} \tilde{u}\|_{L^2(W \times [h_0,1])} \geq C_0 \|f\|_{H^s(W)}.
\end{align*}

Next, we cover $W \times [h_0,1]$ by balls $\{B_{r}(x_k)\}_{k\in \{1,\dots,N\}}\subset \R^{n+1}_+$ of radius $r\in (0,r_0)$ where $r_0:= \frac{h_0}{6}$. Here, the balls are chosen in such a way that all their centres are contained in $W \times [h_0,1]$ and such that if we decrease the radius $r$ of the balls by a factor five the covering becomes disjoint. This is possible by the Vitali covering theorem. 
This choice of the covering also gives a bound of the form $N \leq C(|W|,n) r^{-n}$ for the number of involved balls. Indeed, up to some constant $\bar{C}>1$ we have
\begin{align*}
\bar{C}^{-1}N r^n 5^{-n} \leq |W| \leq \bar{C} N r^n.
\end{align*}
Now, since 
\begin{align*}
\sum\limits_{k=1}^{N} \|x_{n+1}^{\frac{1-2s}{2}} \tilde{u} \|_{L^2(B_{r}(x_k))}^2
\geq \|x_{n+1}^{\frac{1-2s}{2}} \tilde{u}\|_{L^2(W \times [h_0,1])}^2
\geq C_0 \|f\|_{H^{s}(W)}^2,
\end{align*} 
we deduce that at least for one of the balls we have a bound of the type stated in \eqref{eq:bound_balls}.\\

\emph{Step 2b: Chain of balls.}
As explained in \cite[Propositions 5.3 and 5.4]{RS20a} three balls estimates are valid for solutions $\tilde{u}$ to \eqref{eq:CS_ext}, i.e., there exists a constant $\alpha \in (0,1)$ such that for balls $B_{r}(x_0)\subset \R^{n+1}_+$ which also satisfy $B_{4r}(x_0)\subset \R^{n+1}_+$, for solutions $\tilde{u}$ to \eqref{eq:CS_ext} it holds that
\begin{align}
\label{eq:three_balls_1}
\|x_{n+1}^{\frac{1-2s}{2}} \tilde{u} \|_{L^2(B_{r}(x_0))} \leq C \|x_{n+1}^{\frac{1-2s}{2}} \tilde{u} \|_{L^2(B_{r/2}(x_0))}^{\alpha} \|x_{n+1}^{\frac{1-2s}{2}} \tilde{u} \|_{L^2(B_{2r}(x_0))}^{1-\alpha}.
\end{align}
Using the bound $\|x_{n+1}^{\frac{1-2s}{2}} \tilde{u} \|_{L^2(B_{2r}(x_0))} \leq C(|x_0|,r) \|f\|_{H^s(W)}$ (see \eqref{eq:apriori_L2}), this can also be reformulated as
\begin{align}
\label{eq:three_balls_2}
\frac{\|x_{n+1}^{\frac{1-2s}{2}} \tilde{u}\|_{L^2(B_{r}(x_0))}}{\| f \|_{H^s(W)} }
\leq C(|x_0|,r) \left( \frac{\|x_{n+1}^{\frac{1-2s}{2}} \tilde{u}\|_{L^2(B_{r/2}(x_0))}}{\| f \|_{H^s(W)} } \right)^{\alpha}.
\end{align}
Now considering $R>1$ large and starting with a ball $B_{h_0/6}(\hat{x}) \subset B_{R}^+\setminus B_{R/2}^+$, it is possible to connect this ball to the ball $B_{h_0/6}(\bar{x})$ (from Step 2a) through a chain of balls within the upper half plane. More precisely, setting $\tilde{r}=h_0/6$ and $\tilde{x}_0 = \bar{x}$ and $\tilde{x}_N:= \hat{x}$ for some $N\in \N$ we find points $\{\tilde{x}_j\}_{j\in \{1,\dots,N\}} \subset \R^{n+1}_+$ connecting $B_{\tilde{r}}(\tilde{x}_0):=B_{h_0/6}(\bar{x})$ and $B_{\tilde{r}}(\tilde{x}_N):= B_{h_0/6}(\hat{x})$ in such a way that $B_{\tilde{r}/2}(\tilde{x}_j) \subset B_{\tilde{r}}(\tilde{x}_{j+1})$ for $j\in \{0,\dots,N-1\}$ and that the balls of radius $\tilde{r}$ only have controlled overlap.
Recalling that by virtue of Poincar\'e's inequality
\begin{align*}
 \|x_{n+1}^{\frac{1-2s}{2}} \tilde{u} \|_{L^2(B_{r/2}(\hat{x}))} \leq C(x_0,r) \left\| f \right\|_{H^s(W)},
\end{align*}
the inequality \eqref{eq:three_balls_2} implies that for some $\beta = \beta(W, h_0, n):= N \alpha \in (0,1)$ 
\begin{align}
\label{eq:step2_fin}
\begin{split}
\left(\frac{\|f\|_{L^{2}(W)}}{\|f\|_{H^{s}(W)}} \right) 
&\leq C(|x_0|,r,W)
\left(\frac{\|x_{n+1}^{\frac{1-2s}{2}} \tilde{u}\|_{L^2(B_{h_0/6}(\bar{x}))}}{\| f \|_{H^s(W)} }\right)
\leq C(|x_0|,r,W)
\left(\frac{\|x_{n+1}^{\frac{1-2s}{2}} \tilde{u}\|_{L^2(B_{\tilde{r}/2}(\tilde{x}_0))}}{\| f \|_{H^s(W)} }\right)^{\alpha}\\
&\leq C(|x_0|, r,W) \left(\frac{\|x_{n+1}^{\frac{1-2s}{2}} \tilde{u}\|_{L^2(B_{\tilde{r}}(\tilde{x}_1))}}{\| f \|_{H^s(W)} }\right)^{\alpha}
\leq C(|x_0|,r,W)^{\alpha}
\left(\frac{\|x_{n+1}^{\frac{1-2s}{2}} \tilde{u}\|_{L^2(B_{\tilde{r}/2}(\tilde{x}_1))}}{\| f \|_{H^s(W)} }\right)^{\alpha^2 }\\
& \leq \cdots
\leq C(|x_0|,r,W,j,\alpha)
\left(\frac{\|x_{n+1}^{\frac{1-2s}{2}} \tilde{u}\|_{L^2(B_{\tilde{r}/2}(\tilde{x}_j))}}{\| f \|_{H^s(W)} }\right)^{\alpha^j}\\
& \leq \cdots 
\leq C(h_0,W,n,|x_0|,N,\alpha)^N \left( \frac{\|x_{n+1}^{\frac{1-2s}{2}} \tilde{u}\|_{L^2(B_{h_0/6}(\hat{x}))}}{\|f\|_{H^s(W)}} \right)^{ \alpha^N}\\
& \leq  C(h_0,W,n,|x_0|,N,\alpha,h_0)\left( \frac{ \|x_{n+1}^{\frac{1-2s}{2}} \tilde{u}\|_{L^2(B_{R}^+\setminus B_{R/2}^+)}}{\|f\|_{H^s(W)}} \right)^{\beta}.
\end{split}
\end{align}
We remark that in the situation of the annulus $B_{R}^+\setminus B_{R/2}^+$ we have $|x_0| \leq R$.\\

\emph{Step 3: Conclusion.} The argument now follows by combining the estimates \eqref{eq:step1_fin} and \eqref{eq:step2_fin} from Steps 1 and Step 2.
\end{proof}

\subsubsection{Proof of Theorem \ref{thm:doubl_bulk}}
With the auxiliary results from the previous section in hand, we now address the proof of Theorem \ref{thm:doubl_bulk}.

\begin{proof}[Proof of Theorem \ref{thm:doubl_bulk}]
\emph{Step 1: Application of the Carleman estimate from \cite{GRSU18}}.
Without loss of generality, by translation, we may assume that $x_0 = 0$. By scaling we may further assume that $B_{6}' \subset \Omega$.
We then use the Carleman estimate from \cite[Appendix A]{GRSU18}. More precisely, considering the Carleman weight $\phi(x):= \psi(|x|)$ with
\begin{align*}
\psi(r) = - \ln(r) + \frac{1}{10}\left( \ln(r)\arctan(\ln(r)) - \frac{1}{2}\ln(1+ \ln^2(r)) \right),
\end{align*}
there exists constants $C>0$, $\tau_0>0$ such that for any $\tau\geq \tau_0>1$ and 
and for any solution $w$ to
\begin{align*}
\nabla \cdot x_{n+1}^{1-2s} \nabla w & = f \mbox{ in } B_6^+,\\
\lim\limits_{x_{n+1}\rightarrow 0} x_{n+1}^{1-2s} \p_{n+1} w & = q w \mbox{ on } B_6',
\end{align*}
with $\supp(w)\subset B_5^+ \setminus B_{r}^+$ and $r\in (0,4)$ it holds
\begin{align}
\label{eq:Carl}
\begin{split}
&\tau^{\frac{1}{2}} \|e^{\tau \phi} x_{n+1}^{\frac{1-2s}{2}} w\|_{L^2(B_{2r}^+)}
+ \tau^{-\frac{1}{2}} \|e^{\tau \phi} x_{n+1}^{\frac{1-2s}{2}} \nabla w \|_{L^2(B_{r}^+)}\\
& \quad + \tau^s \| e^{\tau \phi} (1+ \ln^2(|x|))^{-\frac{1}{2}}|x|^{-s}w \|_{L^2(B_5')}\\
& \quad + \tau \|e^{\tau \phi}(1+ \ln^2(|x|))^{-\frac{1}{2}} x_{n+1}^{\frac{1-2s}{2}} |x|^{-1} w \|_{L^2(B_5^+)} 
+  \| e^{\tau \phi}(1 + \ln^2(|x|))^{-\frac{1}{2}} x_{n+1}^{\frac{1-2s}{2}} \nabla w \|_{L^2(B_5^+)}\\
& \leq C \tau^{- \frac{1}{2}} \|e^{\tau \phi} |x| x_{n+1}^{\frac{2s-1}{2}} f \|_{L^2(B_5^+)}
+ C \tau^{\frac{1-2s}{2}} \|e^{\tau \phi} |x|^s q w \|_{L^2(B_5')}.
\end{split}
\end{align}
We then argue similarly as in \cite[Lemma 5.4]{GRSU18}. More precisely, we apply the Carleman estimate to $w := \tilde{u} \eta_{r}$, where $\tilde{u}$ is a solution as in \eqref{eq:CS_ext} and $\eta_r$ is a radial cut-off function which localizes at scale $r$ (with gradient and second derivatives controlled by $Cr^{-1}$ and $C r^{-2}$, respectively) and at scale $1$. This gives rise to error contributions on the right hand side of the estimate which are localized at scale $r$ and scale $1$, respectively. We deal with the boundary terms on the right hand side of \eqref{eq:Carl}, by absorbing them into the left hand side by choosing $\tau\geq \tau_0$ (depending on $\|q\|_{L^{\infty}(\Omega)}$) sufficiently large.
Similarly as in \cite[Lemma 5.4]{GRSU18} (but applied with slightly different radii), this leads to the estimate
\begin{align*}
&\tau^{-\frac{1}{2}} e^{\tau \psi(4r)} r^{-1} \|\tilde{u}\|_{H^1_r(B_{4r}^+ \setminus B_{2r}^+, x_{n+1}^{1-2s})} + e^{\tau \psi(5/4)} \|\tilde{u}\|_{H^1(B_{5/2}^+ \setminus B_{2}^+,x_{n+1}^{1-2s})}\\
&\leq C \tau^{-\frac{1}{2}} e^{\tau \psi(3)} \|\tilde{u}\|_{H^1(B_4^+ \setminus B_3^+, x_{n+1}^{1-2s})} + C \tau^{-\frac{1}{2}} r^{-1} e^{\tau \psi(r)} \|  \tilde{u}\|_{H^1_r(B_{2r}^+ \setminus B_{r}^+, x_{n+1}^{1-2s})}.
\end{align*}
Here $\|v\|_{H_r^1(\Omega,x_{n+1}^{1-2s})}:= \|x_{n+1}^{\frac{1-2s}{2}} v \|_{L^2(\Omega)}+ r \|x_{n+1}^{\frac{1-2s}{2}} v \|_{L^2(\Omega)}$ for $\Omega \subset \R^{n+1}_+$ measurable.
By Caccioppoli's inequality, see Proposition \ref{prop:Cacc}, this can further be reduced to an estimate for $L^2$ contributions only:
\begin{align}
\label{eq:CarlL^21}
\begin{split}
&\tau^{-\frac{1}{2}} e^{\tau \psi(4r)} r^{-1} \|\tilde{u}\|_{L^2(B_{4r}^+ \setminus B_{2r}^+, x_{n+1}^{1-2s})} + e^{\tau \psi(5/4)} \|\tilde{u}\|_{L^2(B_{5/2}^+ \setminus B_{2}^+,x_{n+1}^{1-2s})}\\
&\leq C \tau^{-\frac{1}{2}} e^{\tau \psi(3)} \|\tilde{u}\|_{L^2(B_4^+ \setminus B_3^+, x_{n+1}^{1-2s})} + C \tau^{-\frac{1}{2}} r^{-1} e^{\tau \psi(r)} \|  \tilde{u}\|_{L^2(B_{2r}^+ \setminus B_{r}^+, x_{n+1}^{1-2s})}.
\end{split}
\end{align}
Here we used the notation $\|v\|_{L^2(\Omega,x_{n+1}^{1-2s})}:= \|x_{n+1}^{\frac{1-2s}{2}} v\|_{L^2(\Omega)}$ for $\Omega \subset \R^{n+1}_+$ measurable.
Moreover, filling up the ball $B_{2r}^+$ on the left hand side of \eqref{eq:CarlL^21} and the ball $B_{5}^+$ on the right hand side, we obtain
\begin{align}
\begin{split}
\label{eq:CarlL^22}
&\tau^{-\frac{1}{2}} e^{\tau \psi(4r)} r^{-1} \|\tilde{u}\|_{L^2(B_{4r}^+ , x_{n+1}^{1-2s})} + e^{\tau \psi(5/4)} \|\tilde{u}\|_{L^2(B_{5/2}^+ \setminus B_{2}^+,x_{n+1}^{1-2s})}\\
&\leq C \tau^{-\frac{1}{2}} e^{\tau \psi(3)} \|\tilde{u}\|_{L^2(B_5^+ , x_{n+1}^{1-2s})} + C \tau^{-\frac{1}{2}} r^{-1} e^{\tau \psi(r)} \|  \tilde{u}\|_{L^2(B_{2r}^+ , x_{n+1}^{1-2s})}.
\end{split}
\end{align}

\emph{Step 2: Conclusion.}
With \eqref{eq:CarlL^22} at our disposal, we first absorb the term with unit size contributions from the right hand side into the left hand side by choosing $\tau>0$ such that
\begin{align*}
\tau = \frac{1}{\psi(5/2)-\psi(3)} \log\left( 2 \frac{\|\tilde{u}\|_{L^2(B_{5}^+, x_{n+1}^{1-2s})}}{\|\tilde{u}\|_{L^2(B_{5/2}^+ \setminus B_{2}^+, x_{n+1}^{1-2s})}} \right) .
\end{align*}
This yields that
\begin{align*}
\|\tilde{u}\|_{L^2(B_{4r}^+, x_{n+1}^{1-2s})} 
& \leq C \left( 2 \frac{\|\tilde{u}\|_{L^2(B_{5}^+, x_{n+1}^{1-2s})}}{\|\tilde{u}\|_{L^2(B_{5/2}^+ \setminus B_{2}^+, x_{n+1}^{1-2s})}}  \right)^{(\psi(r)-\psi(4r))} \|\tilde{u}\|_{L^2(B_{2r}^+, x_{n+1}^{1-2s})}  \\
& \leq C F^{\gamma (\psi(r)-\psi(4r))}\|\tilde{u}\|_{L^2(B_{2r}^+, x_{n+1}^{1-2s})},
\end{align*}
where we have used the bound from Lemma \ref{lem:annulus}.
Finally, noting that by the properties of $\psi$ (and in particular the logarithm) there exists a constant $C_1>1$ such that
\begin{align*}
C_1^{-1}\leq |\psi(r)-\psi(4r)|\leq C_1, 
\end{align*}
and plugging this into our estimate implies
\begin{align*}
\|\tilde{u}\|_{L^2(B_{4r}^+, x_{n+1}^{1-2s})} 
& \leq C F^{C_1 \gamma }\|\tilde{u}\|_{L^2(B_{2r}^+, x_{n+1}^{1-2s})}.
\end{align*}
This concludes the proof of the bulk doubling inequality.
\end{proof}

\subsection{The boundary doubling inequality}
\label{sec:boundarydoubl}

As a final step in the derivation of the doubling estimates it remains to transfer the doubling estimate from the bulk onto the boundary:

\begin{thm}
\label{thm:doubl_bound}
Let $n\geq 1$ and $s\in (0,1)$. Let $\Omega, W\subset \R^n$ be open, bounded, non-empty Lipschitz sets with $\overline{W}\cap \overline{\Omega} = \emptyset$.  Then, there exists a constant $C>0$ depending only on $\Omega,n , \frac{\|f\|_{H^s(W)}}{\|f\|_{L^2(W)}}$, $\|q\|_{L^{\infty}(\Omega)},s,W$  
such that for any $x_0 \in \Omega$, $r \in (0,r_0)$ with $r_0 \leq \frac{\dist(x_0, \partial \Omega)}{4}$ and for any $\tilde{u}$ solving \eqref{eq:CS_ext} with data $f\in \widetilde{H}^{s}(W)\setminus \{0\}$ it holds
\begin{align*}
\|u\|_{L^2(B_{2r}'(x_0))} \leq C \|u\|_{L^2(B_r'(x_0))}.
\end{align*}
\end{thm}

\begin{proof}
The argument for this is again two-fold: Starting from the bulk doubling estimate, we first bound the right hand side by boundary terms. In a second step, we use trace inequalities in combination with Caccioppoli's estimate to obtain the desired left hand side.\\

\emph{Step 1: Boundary bulk interpolation estimates, bounds for the right hand side.}
We first use a boundary bulk interpolation result where without loss of generality, by scaling, we assume that $B_{4r}'(x_0)\subset \Omega$: For any solution to \eqref{eq:CS_ext} there exists a constant $c_0 = c_0(s,n,W,\Omega, \|q\|_{L^2(\Omega)}) \in (0,1/2)$ such that
\begin{align*}
\|\tilde{u}\|_{L^2(B_{c_0 r}^+(x_0),x_{n+1}^{1-2s})} 
& \leq C \left(\|\tilde{u}\|_{L^2(B_{2r}^+(x_0),x_{n+1}^{1-2s})} + \|u\|_{L^{2}(B_{3 r/2}'(x_0))} \right. \\ 
&\quad \left. + \|\lim\limits_{x_{n+1}\rightarrow 0} x_{n+1}^{1-2s} \p_{n+1} \tilde{u}\|_{L^2(B_{3r/2}'(x_0))} \right)^{\alpha}\times\\
& \quad \times (\|u\|_{L^{2}(B_{3 r/2}'(x_0))} + \|\lim\limits_{x_{n+1}\rightarrow 0} x_{n+1}^{1-2s} \p_{n+1} \tilde{u}\|_{L^2(B_{3r/2}'(x_0))})^{1-\alpha}.
\end{align*}
This was proved in \cite[Proposition 5.13]{RS20a}. 
By the equation in $\Omega$, for $x_0 \in \Omega$ such that also $B_{r}'(x_0) \subset \Omega$ this turns into 
\begin{align}
\label{eq:boundary_bulk1}
\|\tilde{u}\|_{L^2(B_{c_0 r}^+(x_0),x_{n+1}^{1-2s})} \leq C(\|q\|_{L^{\infty}(\Omega)}) (\|\tilde{u}\|_{L^2(B_{2r}^+(x_0),x_{n+1}^{1-2s})} + \|u\|_{L^{2}(B_{3r/2}'(x_0))})^{\alpha}\|u\|_{L^{2}(B_{3r/2}'(x_0))} ^{1-\alpha}.
\end{align}
We now distinguish two cases:
\begin{itemize}
\item In the case that $\|\tilde{u}\|_{L^2(B_{2r}^+(x_0),x_{n+1}^{1-2s})} \leq C \|u\|_{L^{2}(B_{3r/2}'(x_0))}$ for some $C=C(s,W,\Omega)>0$, \eqref{eq:boundary_bulk1} reduces to
\begin{align}
\label{eq:boundary_bulk2}
\|\tilde{u}\|_{L^2(B_{c_0 r}^+(x_0),x_{n+1}^{1-2s})} \leq C(s,W,\Omega,\|q\|_{L^{\infty}(\Omega)})\|u\|_{L^{2}(B_{3r/2}'(x_0))}.
\end{align} 
\item In the case that $\|\tilde{u}\|_{L^2(B_{2r}^+(x_0),x_{n+1}^{1-2s})} > C \|u\|_{L^{2}(B_{3r/2}'(x_0))}$ for the constant $C>0$ from above, we invoke the bulk doubling inequality to nevertheless obtain an estimate of the form \eqref{eq:boundary_bulk2}:
Using the bulk doubling inequality (possibly in an iterated way if $c_0>0$ is small), we first obtain
\begin{align*}
\|\tilde{u}\|_{L^2(B_{2 r}^+(x_0),x_{n+1}^{1-2s})} 
&\leq C(\|q\|_{L^{\infty}(\Omega)},F,c_0) \|\tilde{u}\|_{L^2(B_{c_0 r}^+(x_0),x_{n+1}^{1-2s})} \\
& \leq C(\|q\|_{L^{\infty}(\Omega)},F,c_0) \|\tilde{u}\|_{L^2(B_{2r}^+(x_0),x_{n+1}^{1-2s})}^{\alpha}\|u\|_{L^{2}(B_{3r/2}'(x_0))} ^{1-\alpha}.
\end{align*}
Rearranging this implies that it always holds that
\begin{align}
\label{eq:boundary_bulk2a}
\|\tilde{u}\|_{L^2(B_{2 r}^+(x_0),x_{n+1}^{1-2s})} \leq  C(\|q\|_{L^{\infty}(\Omega)},F,s,W,\Omega) \|u\|_{L^{2}(B_{3r/2}'(x_0))}.
\end{align}
\end{itemize}

\emph{Step 2: Trace and Caccioppoli's estimates, bounds for the left hand side.}
Finally, a trace estimate, together with Caccioppoli's estimate, (iterations of) the bulk doubling estimate and \eqref{eq:boundary_bulk2a} yields
\begin{align}
\label{eq:boundary_bulk3}
\begin{split}
\|u\|_{L^2(B_{2r}'(x_0))} 
&\leq C \|\tilde{u}\|_{H^{1}(B_{2r}'(x_0) \times [0,1], x_{n+1}^{1-2s})}
\leq C(\|q\|_{L^{\infty}(B_{4r}'(x_0))})\|\tilde{u}\|_{L^{2}(B_{4r}'(x_0) \times [0,1], x_{n+1}^{1-2s})}\\
&\leq C(\|q\|_{L^{\infty}(B_{4r}'(x_0))},F)\|\tilde{u}\|_{L^{2}(B_{\frac{4}{3}r}^+(x_0), x_{n+1}^{1-2s})}
\leq   C(\|q\|_{L^{\infty}(\Omega)},F) \|u\|_{L^{2}(B_{r}'(x_0))}.
\end{split}
\end{align}
This then concludes the argument.
\end{proof}

\section{Derivation of the Single Measurement Stability Estimate}
\label{sec:single_meas}

Last but not least, we present the argument for the stability estimate by following the strategy from \cite{ASV13}. To this end, we first recall the boundary interpolation result from \cite{ASV13}:

\begin{lem}
\label{lem:boundaryAVS}
Let $\Omega \subset \R^n$ be an open, bounded, non-empty set. Let $r_0>0$ be so small that $\hat{\Omega}:=\{x\in \Omega: \ \dist(x, \partial \Omega)\geq 2 r_0\} \Subset \Omega$ has non-empty interior. Let $w \in L^2(\Omega)$ be such that for each $x_0 \in \hat{\Omega}$ and $r \in (0,r_0)$ we have $\|w\|_{L^2(B_{r}(x_0))} \geq C_{\text{low}} r^{\beta}$ for some constants $C_{\text{low}}>0$ and $\beta>0$. Let $g \in C^{0,\alpha}(\overline{\Omega})$ for some $\alpha \in (0,1)$ with 
\begin{align*}
|g(x)-g(y)| \leq E |x-y|^{\alpha} \mbox{ for all } x,y \in \overline{\Omega}.
\end{align*}
Assume that there exist constants $C_{\text{stab}}>0, \mu >0$ such that for some $\tilde{E}>0$
\begin{align*}
\left( \int\limits_{\hat{\Omega}} |g|^2 w^2 dx \right)^{\frac{1}{2}} \leq  \frac{C_{\text{stab}}  \tilde{E}}{|\log(\tilde{E}/\epsilon)|^{\mu}}.
\end{align*}
Then there exists a constant $\tilde{C}=\tilde{C}(r_0, \Omega, \hat{\Omega}, C_{\text{low}}, C_{\text{stab}}, n, \beta, \alpha)>0$ such that for all $\epsilon \in (0,1/2)$
\begin{align*}
\|g\|_{L^{\infty}(\hat{\Omega})} \leq C \frac{E^{\frac{\beta}{\alpha + \beta}} \tilde{E}^{\frac{\alpha}{\alpha+ \beta}}}{|\log(\epsilon/\tilde{E})|^{\frac{\alpha}{\alpha+ \beta} \mu}}.
\end{align*}
\end{lem}

\begin{proof}
We follow the ideas from \cite{ASV13} and only present the proof for completeness. Let $\bar{x}_0 \in \Omega$ be such that 
\begin{align*}
|g(\bar{x}_0)| = \|g\|_{L^{\infty}(\hat{\Omega})}. 
\end{align*}
By the assumed Hölder regularity of $g$ we then have that for any $x \in B_{r}'(\bar{x}_0)$
\begin{align*}
|g(\bar{x}_0)| \leq |g(x)| + E r^{\alpha}.
\end{align*}
As a consequence, it also holds that
\begin{align*}
|g(\bar{x}_0)|^2 \int\limits_{B_r'(\bar{x}_0)} w^2 dx \leq \int\limits_{B_r'(\bar{x}_0)} w^2 |g|^2 dx + E^2 r^{2\alpha} \int\limits_{B_r'(\bar{x}_0)} w^2 dx.
\end{align*}
Dividing by $\int\limits_{B_r'(\bar{x}_0)} w^2 dx$ and using the lower bound on its order of vanishing, we arrive at
\begin{align*}
\|g\|_{L^{\infty}(\hat{\Omega})}^2 
& \leq \frac{C_{\text{stab}}^2 \tilde{E}^2}{|\log(\epsilon/ \tilde{E})|^{2\mu}} \frac{1}{\int\limits_{B_r'(\bar{x}_0)}w^2 dx}  + E^2 r^{2\alpha}
\leq C_{\text{stab}}^2 C_{\text{low}}^{-2} r^{-2\beta} \frac{\tilde{E}^2}{|\log(\epsilon/ \tilde{E})|^{2\mu}} + E^2 r^{2\alpha}.
\end{align*}
Optimizing in $r \in (0,1/2)$ we choose
\begin{align*}
r= \min\left\{ \left( \frac{C_{\text{stab}} C_{\text{low}}^{-1} \tilde{E}}{E |\log(\epsilon/\tilde{E})|^{\mu}} \right)^{\frac{1}{\alpha + \beta}}, r_0 \right\}.
\end{align*}
Therefore, there exists a constant $\tilde{C}>0$ with the claimed dependence such that
\begin{align*}
\|g\|_{L^{\infty}(\hat{\Omega})}
\leq \tilde{C} E^{\frac{\beta}{\alpha + \beta}} \tilde{E}^{\frac{\alpha}{\alpha + \beta}} |\log(\epsilon/\tilde{E})|^{-\mu \frac{\alpha}{\alpha + \beta}}.
\end{align*}
This concludes the proof of the claimed stability estimate for $g$.
\end{proof}

Now with Lemma \ref{lem:control_f}, the propagation of smallness result from Proposition \ref{prop:stab} and the boundary doubling estimates from Theorem \ref{thm:doubl_bound} in hand, we provide the proof of Theorem \ref{thm:single_meas_stab}.

\begin{proof}[Proof of Theorem \ref{thm:single_meas_stab}]
We seek to apply Lemma \ref{lem:boundaryAVS}. To this end, we first estimate $u_1(q_1 - q_2)$:
\begin{align}
\label{eq:interpolate}
\|u_1 (q_1 - q_2)\|_{L^2( \Omega)}
\leq C \|u_1 (q_1 - q_2)\|_{H^{s}( \Omega)}^{\frac{1}{2} } \|u_1 (q_1 - q_2)\|_{H^{-s}( \Omega)}^{\frac{1}{2}}.
\end{align}
The first term on the right hand side is controlled by a priori estimates (for $q_j$ and for $u_1$ by the data $f$): Indeed, 
\begin{align*}
\|u_1 (q_1 - q_2)\|_{H^{s}( \Omega)}  
&\leq C \|u_1\|_{H^{s}(\Omega)} (\|q_1\|_{C^{0,s}(\Omega)} + \|q_2\|_{C^{0,s}(\Omega)}) \\
&\leq C \|f\|_{H^s(W)}(\|q_1\|_{C^{0,s}(\Omega)} + \|q_2\|_{C^{0,s}(\Omega)}),
\end{align*}
where we have used the Kato-Ponce inequality \cite{GO14}, the fact that $C^{0,s}(\Omega) \subset H^s(\Omega)$ for $s\in (0,1)$ and the a priori estimate \eqref{eq:apriori_nonloc} for $u_1$.

We thus turn to the second contribution in the interpolation estimate \eqref{eq:interpolate}. For this we use the equation for the Caffarelli-Silvestre extension in $\Omega$:
\begin{align*}
c_s \lim\limits_{x_{n+1} \rightarrow 0} x_{n+1}^{1-2s} \p_{n+1} \tilde{u}_j + q_j u_j = 0 \mbox{ in } \Omega.
\end{align*}
(which holds in an $H^{-s}(\Omega)$ weak sense).
Hence, using that 
\begin{align*}
u_1(q_1 - q_2) &= u_1 q_1 - u_2 q_2 - q_2 (u_1 -u_2)\\
&= c_s \lim\limits_{x_{n+1}\rightarrow 0} x_{n+1}^{1-2s} \p_{n+1} \tilde{u}_1 - c_s \lim\limits_{x_{n+1} \rightarrow 0} x_{n+1}^{1-2s} \p_{n+1} \tilde{u}_2 - q_2 (u_1- u_2),
\end{align*}
we infer that
\begin{align*}
\|u_1(q_1 - q_2)\|_{H^{-s}(\Omega)} 
& \leq 
C \left( c_s \left\| \lim\limits_{x_{n+1}\rightarrow 0} x_{n+1}^{1-2s} \p_{n+1} \tilde{u}_1 - \lim\limits_{x_{n+1} \rightarrow 0} x_{n+1}^{1-2s} \p_{n+1} \tilde{u}_2 \right\|_{H^{-s}(\Omega)} \right.\\
& \left. \quad + \|q_2 (u_1 - u_2)\|_{H^{-s}(\Omega)}
 \right)\\
 & \leq C(\|q_2\|_{C^{0,s}(\Omega)},s) \left( \left\| \lim\limits_{x_{n+1}\rightarrow 0} x_{n+1}^{1-2s} \p_{n+1} \tilde{u}_1 - \lim\limits_{x_{n+1} \rightarrow 0} x_{n+1}^{1-2s} \p_{n+1} \tilde{u}_2 \right\|_{H^{-s}(\Omega)} \right. \\
 & \left. \quad
+ \|u_1 - u_2\|_{H^{-s}(\Omega)}
 \right)\\
 & \leq C(\|q_1\|_{C^{0,s}(\Omega)}, \|q_2\|_{C^{0,s}(\Omega)},s) \frac{C\|f\|_{H^{s+\epsilon}(W)}}{\left|\log\left(C\frac{\|f\|_{H^{s+\epsilon}(W)}}{\|(\Lambda_{q_1}-\Lambda_{q_2})f\|_{H^{-s}(W)}} \right) \right|^{\mu}}.
\end{align*}
Combined with the previous estimates this entails that 
\begin{align}
\label{eq:est_domain}
\|u_1 (q_1 - q_2)\|_{L^2( \Omega)}
\leq C(\|q_1\|_{C^{0,s}(\Omega)}, \|q_2\|_{C^{0,s}(\Omega)},s ) \frac{C\|f\|_{H^{s+\epsilon}(W)}}{\left|\log\left(C\frac{\|f\|_{H^{s+\epsilon}(W)}}{\|(\Lambda_{q_1}-\Lambda_{q_2})f\|_{H^{-s}(W)}} \right) \right|^{\mu/2}}.
\end{align}

Finally, we note that for any $x_0\in \Omega$ and any $r\in (0,r_0)$ with $\dist(x_0, \partial \Omega) \geq r_0 >0$ by an iteration of the boundary doubling inequality from Theorem \ref{thm:doubl_bound} we have
\begin{align*}
\int\limits_{B_{r}(x_0)} u^2_1 dx \geq C r^{\beta}
\end{align*} 
for some constants $C>0$ and $\beta>0$ depending on $W,n,\Omega, \|q_1\|_{L^{\infty}(\Omega)}, \|q_2\|_{L^{\infty}(\Omega)}$. 
Combining this, the fact that $u_j \in H^{s}(\R^n)$ for $j\in\{1,2\}$ and \eqref{eq:est_domain} with the observation of Lemma \ref{lem:boundaryAVS} and the assumption that $\supp(q_j) \subset \Omega' \Subset \Omega$, then concludes the proof of Theorem \ref{thm:single_meas_stab}.
\end{proof}

\bibliographystyle{alpha}
\bibliography{citationsFT}

\begin{thebibliography}{GRSU20}

\bibitem[ASV13]{ASV13}
Giovanni Alessandrini, Eva Sincich, and Sergio Vessella.
\newblock Stable determination of surface impedance on a rough obstacle by far
  field data.
\newblock {\em Inverse Problems \& Imaging}, 7(2):341, 2013.

\bibitem[BL15]{BL15}
Katar{\'\i}na Bellov{\'a} and Fang-Hua Lin.
\newblock Nodal sets of {S}teklov eigenfunctions.
\newblock {\em Calculus of Variations and Partial Differential Equations},
  54(2):2239--2268, 2015.

\bibitem[Cal06]{Calderon}
Alberto~P Calder{\'o}n.
\newblock On an inverse boundary value problem.
\newblock {\em Computational \& Applied Mathematics}, 25(2-3):133--138, 2006.

\bibitem[CLR20]{CLR20}
Mihajlo Ceki{\'c}, Yi-Hsuan Lin, and Angkana R{\"u}land.
\newblock The {C}alder{\'o}n problem for the fractional {S}chr{\"o}dinger
  equation with drift.
\newblock {\em Calculus of Variations and Partial Differential Equations},
  59:1--46, 2020.

\bibitem[CR20]{CR20}
Giovanni Covi and Angkana R{\"u}land.
\newblock On some partial data {C}alder{\'o}n type problems with mixed boundary
  conditions.
\newblock {\em arXiv preprint arXiv:2006.03252}, 2020.

\bibitem[CS07]{CS07}
Luis Caffarelli and Luis Silvestre.
\newblock An extension problem related to the fractional {L}aplacian.
\newblock {\em Communications in partial differential equations},
  32(8):1245--1260, 2007.

\bibitem[DSV17]{DSV14}
Serena Dipierro, Ovidiu Savin, and Enrico Valdinoci.
\newblock All functions are locally {$ s $}-harmonic up to a small error.
\newblock {\em Journal of the European Mathematical Society}, 19(4):957--966,
  2017.

\bibitem[FF14]{FF14}
Mouhamed~Moustapha Fall and Veronica Felli.
\newblock Unique continuation property and local asymptotics of solutions to
  fractional elliptic equations.
\newblock {\em Communications in Partial Differential Equations},
  39(2):354--397, 2014.

\bibitem[GFR19]{GFR19}
Mar{\'i}a~{\'A}ngeles Garc{\'i}a-Ferrero and Angkana R{\"u}land.
\newblock Strong unique continuation for the higher order fractional
  {L}aplacian.
\newblock {\em Mathematics in Engineering}, 1(mine-01-04-715):715, 2019.

\bibitem[GFR20]{GFR20}
Mar{\'\i}a~{\'A}ngeles Garc{\'\i}a-Ferrero and Angkana R{\"u}land.
\newblock On two methods for quantitative unique continuation results for some
  nonlocal operators.
\newblock {\em Comm. PDE}, 45(11):1512--1560, 2020.

\bibitem[GLX17]{GLX17}
Tuhin Ghosh, Yi-Hsuan Lin, and Jingni Xiao.
\newblock The {C}alder{\'o}n problem for variable coefficients nonlocal
  elliptic operators.
\newblock {\em Communications in Partial Differential Equations},
  42(12):1923--1961, 2017.

\bibitem[GO14]{GO14}
Loukas Grafakos and Seungly Oh.
\newblock The {K}ato-{P}once inequality.
\newblock {\em Communications in Partial Differential Equations},
  39(6):1128--1157, 2014.

\bibitem[GRSU20]{GRSU18}
Tuhin Ghosh, Angkana R{\"u}land, Mikko Salo, and Gunther Uhlmann.
\newblock Uniqueness and reconstruction for the fractional {C}alder{\'o}n
  problem with a single measurement.
\newblock {\em Journal of Functional Analysis}, page 108505, 2020.

\bibitem[GSU20]{GSU16}
Tuhin Ghosh, Mikko Salo, and Gunther Uhlmann.
\newblock The {C}alder{\'o}n problem for the fractional {S}chr{\"o}dinger
  equation.
\newblock {\em Anal. PDE}, 13(2):455--475, 2020.

\bibitem[HL19]{HL19}
Bastian Harrach and Yi-Hsuan Lin.
\newblock Monotonicity-based inversion of the fractional {S}chro{\"o}dinger
  equation {I}. {P}ositive potentials.
\newblock {\em SIAM Journal on Mathematical Analysis}, 51(4):3092--3111, 2019.

\bibitem[HL20]{HL20}
Bastian Harrach and Yi-Hsuan Lin.
\newblock Monotonicity-based inversion of the fractional {S}chr{\"o}dinger
  equation {I}{I}. {G}eneral potentials and stability.
\newblock {\em SIAM Journal on Mathematical Analysis}, 52(1):402--436, 2020.

\bibitem[LLR20]{LLR20}
Ru-Yu Lai, Yi-Hsuan Lin, and Angkana R{\"u}land.
\newblock The {C}alder{\'o}n problem for a space-time fractional parabolic
  equation.
\newblock {\em SIAM Journal on Mathematical Analysis}, 52(3):2655--2688, 2020.

\bibitem[RS18]{RS18}
Angkana R{\"u}land and Mikko Salo.
\newblock Exponential instability in the fractional {C}alder{\'o}n problem.
\newblock {\em Inverse Problems}, 34(4):045003, 2018.

\bibitem[RS19a]{RS19}
Angkana R{\"u}land and Mikko Salo.
\newblock Quantitative {R}unge approximation and inverse problems.
\newblock {\em International Mathematics Research Notices},
  2019(20):6216--6234, 2019.

\bibitem[RS19b]{RueS19}
Angkana R{\"u}land and Eva Sincich.
\newblock Lipschitz stability for the finite dimensional fractional
  {C}alder{\'o}n problem with finite {C}auchy data.
\newblock {\em Inverse Problems and Imaging}, 13(5):1023--1044, 2019.

\bibitem[RS20a]{RS20a}
Angkana R{\"u}land and Mikko Salo.
\newblock The fractional {C}alder{\'o}n problem: low regularity and stability.
\newblock {\em Nonlinear Analysis}, 193:111529, 2020.

\bibitem[RS20b]{RS20}
Angkana R{\"u}land and Mikko Salo.
\newblock Quantitative approximation properties for the fractional heat
  equation.
\newblock {\em Mathematical Control \& Related Fields}, 10(1):1, 2020.

\bibitem[R{\"u}l15]{Rue15}
Angkana R{\"u}land.
\newblock Unique continuation for fractional {S}chr{\"o}dinger equations with
  rough potentials.
\newblock {\em Communications in Partial Differential Equations},
  40(1):77--114, 2015.

\bibitem[R{\"u}l17]{Rue17}
Angkana R{\"u}land.
\newblock On quantitative unique continuation properties of fractional
  {S}chr{\"o}dinger equations: Doubling, vanishing order and nodal domain
  estimates.
\newblock {\em Transactions of the American Mathematical Society},
  369(4):2311--2362, 2017.

\bibitem[R{\"u}l18]{Rue18}
Angkana R{\"u}land.
\newblock Unique continuation, {R}unge approximation and the fractional
  {C}alder{\'o}n problem.
\newblock {\em Journ{\'e}es {\'e}quations aux d{\'e}riv{\'e}es partielles},
  pages 1--10, 2018.

\bibitem[R{\"u}l19]{R19}
Angkana R{\"u}land.
\newblock Quantitative invertibility and approximation for the truncated
  {H}ilbert and {R}iesz transforms.
\newblock {\em Revista Matem{\'a}tica Iberoamericana}, 35(7):1997--2024, 2019.

\bibitem[RW19]{RW19}
Angkana R{\"u}land and Jenn-Nan Wang.
\newblock On the fractional {L}andis conjecture.
\newblock {\em Journal of Functional Analysis}, 277(9):3236--3270, 2019.

\bibitem[Sal17]{S17}
Mikko Salo.
\newblock The fractional {C}alder{\'o}n problem.
\newblock {\em Journ{\'e}es {\'e}quations aux d{\'e}riv{\'e}es partielles},
  pages 1--8, 2017.

\bibitem[Sin10]{S10}
Eva Sincich.
\newblock Stability for the determination of unknown boundary and impedance
  with a {R}obin boundary condition.
\newblock {\em SIAM Journal on Mathematical Analysis}, 42(6):2922--2943, 2010.

\bibitem[Uhl09]{U09}
Gunther Uhlmann.
\newblock Electrical impedance tomography and {C}alder{\'o}n's problem.
\newblock {\em Inverse problems}, 25(12):123011, 2009.

\bibitem[V{\`E}65]{VE65}
Marko~Iosifovich Vishik and Gregory~Il'ich {\`E}skin.
\newblock Equations in convolutions in a bounded region.
\newblock {\em Russian Mathematical Surveys}, 20(3):85, 1965.

\bibitem[Yu17]{Y17}
Hui Yu.
\newblock Unique continuation for fractional orders of elliptic equations.
\newblock {\em Annals of PDE}, 3(2):16, 2017.

\end{thebibliography}

\end{document}